\documentclass[12pt, oneside, reqno]{article}
\usepackage{amsmath,amsthm}
\usepackage{amssymb,latexsym}
\usepackage{enumerate}
\usepackage{cases}

\usepackage{tikz}
\newtheorem{thm}{Theorem}[section]

\newtheorem{lem}[thm]{Lemma}

\theoremstyle{definition}

\numberwithin{equation}{section}

\begin{document}

\title{\Large On a subset sums problem of Chen and Wu
}\author{\large Min Tang\thanks{Corresponding author. This work was supported by the National Natural Science Foundation of China(Grant
No. 11971033) and top talents project of Anhui Department of Education(Grant No. gxbjZD05).} and Hongwei Xu}
\date{} \maketitle
 \vskip -3cm
\begin{center}
\vskip -1cm { \small
\begin{center}
 School of Mathematics and Statistics, Anhui Normal
University
\end{center}
\begin{center}
Wuhu 241002, PR China
\end{center}}
\end{center}

  {\bf Abstract:} For a set $A$, let $P(A)$ be the set of all finite subset sums of $A$. We prove that if a sequence $B=\{11\leq b_1<b_2<\cdots\}$ satisfies  $b_2=3b_1+5$, $b_3=3b_2+2$ and $b_{n+1}=3b_n+4b_{n-1}$ for all $n\geq 3$, then there is a sequence of positive integers $A=\{a_1<a_2<\cdots\}$ such that
 $P(A)=\mathbb{N}\setminus B$. This result shows that the answer to the problem of Chen and Wu [`The inverse problem on subset sums', European. J. Combin. 34(2013), 841-845] is negative.


\vskip8mm

\section{Introduction}
Let $\mathbb{N}$ be the set of all nonnegative integers. For a sequence of integers $A=\{a_1<a_2<\cdots\}$, let
$$P(A)=\left\{\sum \varepsilon_ia_i: a_i\in A, \varepsilon_i=0\text{ or }1, \sum \varepsilon_i<\infty\right\}.$$
Here $0\in P(A)$.

 In 1970, S. A. Burr \cite{Burr} asked the following question: which subsets $S$ of $\mathbb{N}$ are equal to $P(A)$ for some $A$?
Burr showed the following result (unpublished):

\noindent{\bf Theorem A} (\cite{Burr}).
{\it Let $B=\{4\leq b_1<b_2<\cdots\}$ be a sequence of integers for which $b_{n+1}\geq b_n^2$ for $n=1,2,\ldots$. Then there exists $A=\{a_1<a_2<\cdots\}$ for which
 $P(A)=\mathbb{N}\setminus B$.}

 Burr \cite{Burr} ever mentioned that if $B$ grows sufficiently rapidly, then there exists a sequence $A$ such that  $P(A)=\mathbb{N}\setminus B$.
 It is natural to ask how slow can sequence $B$ grow. In 1996, Hegyv\'{a}ri \cite{Hegy} improved Burr's result:

 \noindent{\bf Theorem B} (\cite{Hegy}, Theorem 1).
{\it Let $B=\{7\leq b_1<b_2<\cdots\}$ be a sequence of integers. Suppose that for every $n$, $b_{n+1}\geq 5b_n$. Then there exists a sequence of integers $A=\{a_1<a_2<\cdots\}$ for which
 $P(A)=\mathbb{N}\setminus B$.}

In 2012, Chen and Fang \cite{chen2012} precisely extended Hegyv\'{a}ri's result related to Burr¡¯s question:

\noindent{\bf Theorem C} (\cite{chen2012}, Theorem 1).
{\it Let $B=\{b_1<b_2<\cdots\}$ be a sequence of integers with $b_1\in\{4,7,8\}\cup \{b: b\geq 11, b\in \mathbb{N}\}$ and $b_{n+1}\geq 3b_n+5$ for all $n\geq 1$. Then there exists a sequence of positive integers $A=\{a_1<a_2<\cdots\}$ for which
 $P(A)=\mathbb{N}\setminus B$.}

\noindent{\bf Theorem D} (\cite{chen2012}, Theorem 2).
{\it Let $B=\{b_1<b_2<\cdots\}$ be a sequence of positive integers with $b_1\in\{3,5,6,9,10\}$ or $b_2=3b_1+4$ or $b_1=1$, $b_2=9$ or $b_1=2$, $b_2=15$. Then there is no sequence of positive integers $A=\{a_1<a_2<\cdots\}$ for which
 $P(A)=\mathbb{N}\setminus B$.}

 In 2013, Chen and Wu \cite{chen2013} further improved Theorem C.

 \noindent{\bf Theorem E} (\cite{chen2013}, Theorem 1).
{\it If $B=\{b_1<b_2<\cdots\}$ is a sequence of integers with $b_1\in\{4,7,8\}\cup \{b: b\geq 11, b\in \mathbb{N}\}$, $b_2\geq 3b_1+5$, $b_3\geq 3b_2+3$ and $b_{n+1}> 3b_n-b_{n-2}$ for all $n\geq 3$, then there exists a sequence of positive integers $A=\{a_1<a_2<\cdots\}$ such that
 $P(A)=\mathbb{N}\setminus B$ and
 $$P(A_s)=[0,2b_s]\setminus \{b_1,\ldots, b_s, 2b_s-b_{s-1}, \ldots,2b_s-b_1\},$$
 where $A_s=A\cap [0, b_s-b_{s-1}]$ for all $s\geq 2$.}

 Moreover, Chen and Wu \cite{chen2013} posed the following problem:

  \noindent{\bf Problem 1} (\cite{chen2013}, Problem 1). {\it Let $B=\{b_1<b_2<\cdots\}$ be a sequence of positive integers. Let $d_1=10$, $d_2=3b_1+4$, $d_3=3b_2+2$ and $d_{n+1}=3b_n-b_{n-2}(n\geq 3)$. If $b_m=d_m$ for some $m\geq 3$ and $b_n>d_n$ for all $n\neq m$. Is it true there is no sequence of positive integers $A=\{a_1<a_2<\cdots\}$ with $P(A)=\mathbb{N}\setminus B$?}

In 2013, Wu \cite{Wu2013} gave a segment version of Theorem E. Recently, Fang and Fang \cite{Fang2019} determined the critical value for $b_3$ such that there exists an infinite sequence of positive integers $A$ for which
 $P(A)=\mathbb{N}\setminus B$(under the condition $b_1>1$ and $b_2=3b_1+5$).

 In this paper, we obtain the following results:

\begin{thm}\label{thm0} Let $B=\{11\leq b_1<b_2<\cdots\}$ be a sequence of integers with $b_2=3b_1+5$, $b_3=3b_2+2$ and $b_{n+1}=3b_n+4b_{n-1}$ for all $n\geq 3$. Then there exists a sequence of positive integers $A=\{a_1<a_2<\cdots\}$ such that
 $P(A)=\mathbb{N}\setminus B$.
\end{thm}

\noindent{\bf Remark 1} Theorem \ref{thm0} shows that the answer to Problem 1 is negative for $m=3$.

 \begin{thm}\label{thm1} Let $B=\{3\leq b_1<b_2<\cdots\}$ be a sequence of integers. If $b_2=3b_1+3$ or $b_2=3b_1+2$, then there is no sequence of positive integers $A=\{a_1<a_2<\cdots\}$ such that
 $P(A)=\mathbb{N}\setminus B$.
\end{thm}

\begin{thm}\label{thm2} Let $B=\{b_1<b_2<\cdots\}$ be an infinite arithmetic progression with common difference $d$ and $b_1\in\{4,7,8\}\cup \{b: b\geq 11, b\in \mathbb{N}\}$. If $b_1+2\leq d\leq 2b_1+1$, then there exists a sequence of positive integers $A=\{a_1<a_2<\cdots\}$ such that
 $P(A)=\mathbb{N}\setminus B$.
\end{thm}
\noindent{\bf Remark 2} Theorem \ref{thm2} shows that there exists a sequence of positive integers $A=\{a_1<a_2<\cdots\}$ such that
 $P(A)=\mathbb{N}\setminus B$ for a given special common difference sequence with $2b_1+2\leq b_2\leq 3b_1+1$. Theorem \ref{thm1} and
Theorem \ref{thm2} further enrich our understanding of Burr's problem.

\section{Lemma}

\begin{lem}\label{lem2} Let $A=\{a_1<a_2<\cdots\}$ and $B=\{b_1<b_2<\cdots\}$ be two sequences of positive integers with $b_1>1$ such that $P(A)=\mathbb{N}\setminus B$. If $P(\{a_1,\ldots,a_k\})=[0,b_1-1]$ and $b_2\geq 2b_1+2$, then $a_{k+1}=b_1+1$, $a_{k+2}\leq 2b_1+1$ and
$$P(\{a_1,\ldots,a_{k+1}\})=[0, 2b_1]\setminus \{b_1\},$$
$$P(\{a_1,\ldots,a_{k+2}\})=[0, a_{k+2}+2b_1] \setminus \{b_1, a_{k+2}+b_1\}.$$
\end{lem}
\begin{proof} Since $$P(\{a_1,\ldots,a_k\})=[0, b_1-1],$$
$b_1+1\in P(A)$ and $b_1\not\in P(A)$, we have $a_{k+1}=b_1+1$.
Hence $$P(\{a_1,\ldots,a_{k+1}\})=[0, b_1-1]\cup [a_{k+1},a_{k+1}+b_1-1]=[0, 2b_1]\setminus \{b_1\},$$
and $$a_{k+2}+P(\{a_1,\ldots,a_{k+1}\})=[a_{k+2}, a_{k+2}+2b_1]\setminus \{a_{k+2}+b_1\}.$$

If $a_{k+2}\geq 2b_1+2$, then $2b_1+1\not\in P(A)$ and $b_2=2b_1+1$, which contradicts with $b_2\geq 2b_1+2$. So $a_{k+2}\leq 2b_1+1$ and $$P(\{a_1,\ldots,a_{k+2}\})=[0, a_{k+2}+2b_1] \setminus \{b_1, a_{k+2}+b_1\}.$$

This completes the proof of Lemma \ref{lem2}.
\end{proof}

\section{Proof of Theorem \ref{thm0}}
First, we shall prove that the following result:

\noindent{\bf Fact I} {\it There exists a set sequence $\{A_k\}_{k=3}^{\infty}$ such that

\noindent(i) $A_k=A_{k-1}\cup \{b_{k-1}+2b_{k-3}, b_{k-1}+b_{k-2}-b_{k-3}, b_{k-1}+2b_{k-2}-b_{k-3}\}$ for $k\geq 4$;

\noindent(ii) $P(A_k)=[0, b_k+b_{k-1}]\setminus \{b_1,\ldots,b_k, b_k+b_{k-1}-b_i: \; i=1,\ldots,k-2\}$ for $k\geq 4$;

\noindent(iii) $b_{k}=3b_{k-1}+4b_{k-2}$ for $k\geq 4$.}
\vskip 3mm
By the proof of [2, Theorem 1], there exists $A_1=\{a_1<a_2<\ldots<a_k\}\subseteq [1,b_1-1]$ such that
 $$P(A_1)=[0,b_1-1].$$
By Lemma \ref{lem2}, we have $a_{k+1}=b_1+1$ and $$P(A_1\cup \{b_1+1\})=[0, 2b_1]\setminus \{b_1\}.$$
Let $a_{k+2}=b_1+2,\; a_{k+3}=b_1+3.$ Then $$P(\{a_1,\ldots,a_{k+3}\})=[0, b_1+b_2] \setminus \{b_1, b_2\}.$$
Let $a_{k+4}=b_1+b_2$, $a_{k+5}=2b_2-2b_1+2$. Then
$$b_1+b_2+P(\{a_1,\ldots,a_{k+3}\})=[b_1+b_2,2b_1+2b_2]\backslash \{2b_1+b_2, b_1+2b_2\},$$
thus by $b_3=3b_2+2$, we have
$$P(\{a_1,\ldots,a_{k+4}\})=[0,2b_1+2b_2]\backslash \{b_1,b_2,2b_1+b_2, b_1+2b_2\},$$
$$a_{k+5}+P(\{a_1,\ldots,a_{k+4}\})=[2b_2-2b_1+2,b_3+b_2]\setminus \mathcal{B}_0,$$ where $\mathcal{B}_0=\{2b_2-b_1+2,3b_2-2b_1+2,b_3, b_3+b_2-b_1\}.$

Write
$$A_3=A_1\cup\{b_1+1, b_1+2,b_1+3\}\cup \{b_1+b_2, 2b_2-2b_1+2\}.$$
Since $$2b_2-2b_1+2<2b_1+b_2<2b_2-b_1+2<b_1+2b_2<3b_2-2b_1+2<2b_1+2b_2,$$
 we have$$P(A_3)=[0, b_3+b_2] \setminus \{b_1, b_2,b_3,b_3+b_2-b_1\}.$$

Noting that
$$\max A_3=2b_2-2b_1+2<b_3+2b_1<b_3+b_2-b_1<b_3+2b_2-b_1,$$and
$$ b_3+b_2-b_1+P(A_3)=[b_3+b_2-b_1, 2b_3+2b_2-b_1]\setminus \mathcal{B}_{3,1},$$
where
$$\mathcal{B}_{3,1}=\Big\{b_3+b_2, b_3+2b_2-b_1,
2b_3+b_2-b_1,2b_3+2b_2-2b_1\Big\}.$$
Thus
$$P(A_3\cup\{b_3+b_2-b_1\})=[0,2b_3+2b_2-b_1]\setminus \mathcal{B}_{3,2},$$
where $$\mathcal{B}_{3,2}=\Big\{b_1,b_2, b_3, b_3+2b_2-b_1,
2b_3+b_2-b_1,2b_3+2b_2-2b_1\Big\}.$$

Similarly, we have $$P(A_3\cup\{b_3+b_2-b_1, b_3+2b_2-b_1\})=[0,3b_3+4b_2-2b_1]\setminus \mathcal{B}_{3,4},$$
where $$\mathcal{B}_{3,4}=\Big\{b_1,b_2, b_3, 2b_3+4b_2-2b_1,
3b_3+3b_2-2b_1,3b_3+4b_2-3b_1\Big\}.$$

Noting that
$$b_3+2b_1+P(A_3\cup\{b_3+b_2-b_1,b_3+2b_2-b_1\})=[b_3+2b_1, 4b_3+4b_2]\setminus \mathcal{B}_{3,5},$$
where $$\mathcal{B}_{3,5}=\Big\{b_3+3b_1, b_3+b_2+2b_1,2b_3+2b_1,
3b_3+4b_2,4b_3+3b_2, 4b_3+4b_2-b_1\Big\}.$$
Let  \begin{equation}\label{eq1}b_4=3b_3+4b_2,\end{equation}
\begin{equation}\label{eq2}A_4=A_3\cup \{b_3+b_2-b_1,b_3+2b_2-b_1,b_3+2b_1\}.\end{equation}
Then \begin{equation}\label{eq3}P(A_4)=[0, b_4+b_3]\setminus \{b_1,b_2, b_3,b_4, b_4+b_3-b_2, b_4+b_3-b_1\}.\end{equation}

By (\ref{eq1}), (\ref{eq2}), (\ref{eq3}), we know that Fact I is true for $k=4$. Suppose that Fact I is true for $k(\geq 4)$. Now we consider the case $k+1$.

Since $$A_{k}=A_{k-1}\cup \{b_{k-1}+b_{k-2}-b_{k-3},b_{k-1}+2b_{k-2}-b_{k-3},b_{k-1}+2b_{k-3}\},$$
$$P(A_k)=[0, b_k+b_{k-1}]\setminus \{b_1,\ldots,b_k, b_k+b_{k-1}-b_i: \; i=1,\ldots,k-2\},$$
and
$$\max A_k=b_{k-1}+2b_{k-2}-b_{k-3}<b_k+2b_{k-2}<b_k+b_{k-1}-b_{k-2}<b_k+2b_{k-1}-b_{k-2},$$
we have
$$b_k+b_{k-1}-b_{k-2}+P(A_k)=[b_k+b_{k-1}-b_{k-2}, 2b_k+2b_{k-1}-b_{k-2}]
\setminus \mathcal{B}_{k,1},$$
where $$\mathcal{B}_{k,1}=\Big\{b_k+b_{k-1}-b_{k-2}+b_i,2b_k+2b_{k-1}-b_{k-2}-b_i: \; i=1,\ldots,k-1\Big\}.$$
Noting that $$\begin{array}{ll}&{\bf b_k+b_{k-1}-b_{k-2}}<b_k+b_{k-1}-b_{k-2}+b_1<\cdots<b_k+b_{k-1}-b_{k-2}+b_{k-3}\\
<&{\bf b_k+b_{k-1}-b_{k-3}}<\cdots<b_k+b_{k-1}-b_1\\
<&{\bf b_k+b_{k-1}=b_k+b_{k-1}-b_{k-2}+b_{k-2}},\end{array}$$
we have $$P(A_k\cup \{b_k+b_{k-1}-b_{k-2}\})=[0, 2b_k+2b_{k-1}-b_{k-2}]
\setminus \mathcal{B}_{k,2},$$
where $$\mathcal{B}_{k,2}=\Big\{b_1,\ldots,b_k,2b_k+2b_{k-1}-b_{k-2}-b_i: i=1,\ldots,k\Big\}.$$
Noting that $$\begin{array}{ll}&b_k+2b_{k-1}-b_{k-2}+P(A_k\cup \{b_k+b_{k-1}-b_{k-2}\})\\
=&[b_k+2b_{k-1}-b_{k-2}, 3b_k+4b_{k-1}-2b_{k-2}]
\setminus \mathcal{B}_{k,3},\end{array}$$
where $$\mathcal{B}_{k,3}=\Big\{b_k+2b_{k-1}-b_{k-2}+b_i,3b_k+4b_{k-1}-2b_{k-2}-b_i: \; i=1,\ldots,k\Big\}.$$
Noting that $$\begin{array}{ll}&{\bf b_k+2b_{k-1}-b_{k-2}}<b_k+2b_{k-1}-b_{k-2}+b_1<\cdots<b_k+3b_{k-1}-b_{k-2}\\
<&{\bf 2b_k+b_{k-1}-b_{k-2}}<\cdots<2b_k+2b_{k-1}-b_{k-2}-b_1\\
<&{\bf 2b_k+2b_{k-1}-b_{k-2}},\end{array}$$
we have $$\begin{array}{ll}&P(A_k\cup \{b_k+b_{k-1}-b_{k-2},b_k+2b_{k-1}-b_{k-2}\})\\=&[0, 3b_k+4b_{k-1}-2b_{k-2}]
\setminus \mathcal{B}_{k,4},\end{array}$$
where $$\mathcal{B}_{k,4}=\Big\{b_1,\ldots, b_k, 3b_k+4b_{k-1}-2b_{k-2}-b_i:  i=1,\ldots,k\Big\}.$$
Noting that $$\begin{array}{ll}&b_k+2b_{k-2}+P(A_k\cup \{b_k+b_{k-1}-b_{k-2},b_k+2b_{k-1}-b_{k-2}\})\\
=&[b_k+2b_{k-2}, 4b_k+4b_{k-1}]
\setminus \mathcal{B}_{k,5},\end{array}$$
where $$\mathcal{B}_{k,5}=\Big\{b_k+2b_{k-2}+b_i,4b_k+4b_{k-1}-b_i: \; i=1,\ldots,k\Big\}.$$
Noting that $$\begin{array}{ll}&{\bf b_k+2b_{k-2}}<b_k+2b_{k-2}+b_1<\cdots<2b_k+2b_{k-2}\\
<&{\bf 2b_k+4b_{k-1}-2b_{k-2}}<\cdots<{\bf 3b_k+4b_{k-1}-2b_{k-2}-b_1},\end{array}$$
we have $$P(A_k\cup \{b_k+b_{k-1}-b_{k-2},b_k+2b_{k-1}-b_{k-2},b_k+2b_{k-2}\})=[0, 4b_k+4b_{k-1}]
\setminus \mathcal{B}_{k,6},$$
where $$\mathcal{B}_{k,6}=\Big\{b_1,\ldots, b_k, 4b_k+4b_{k-1}-b_i: \; i=1,\ldots,k\Big\}.$$

Write $$b_{k+1}=3b_k+4b_{k-1},$$ $$A_{k+1}=A_{k}\cup \{b_k+b_{k-1}-b_{k-2},b_k+2b_{k-1}-b_{k-2},b_k+2b_{k-2}\},$$
we have $$P(A_{k+1})=[0,b_{k+1}+b_k]\setminus \{b_1,\ldots,b_{k+1},b_{k+1}+b_k-b_i: \; i=1,\ldots,k-1\}.$$
Second, let $$A=\bigcup_{k=4}^{\infty}A_k.$$ If $n\in P(A)$, let $n<b_k+2b_{k-2}$, then noting that
$$A\setminus A_i\subseteq [b_{k}+2b_{k-2},+\infty)$$
for all $i\geq k$,
we have $n\in P(A_{k})$.

By Fact I (ii) we have \begin{equation}\label{eq4}n\not\in \{b_1,\ldots,b_k,b_k+b_{k-1}-b_i: \; i=1,\ldots,k-2\}.\end{equation}
If $n\leq b_k$, then by (\ref{eq4}) we have $n\not\in B$.
If $b_k<n<b_k+2b_{k-2}$, then by $b_k<n<b_{k+1}$, we have $n\not\in B$.
 That is, $n\in \mathbb{N}\setminus B$.

Conversely, if $n'\in \mathbb{N}\setminus B$, then $n'\not\in B$, let $n'<b_{k'}$, we have $$n'\not\in \{b_1,\ldots,b_{k'}, b_{k'}+b_{k'-1}-b_i: \; i=1,\ldots,k'-2\}.$$ By Fact I (ii) we have $n'\in P(A_{k'})$. So $n'\in P(A)$.

Hence $P(A)=\mathbb{N}\setminus B$.

This completes the proof of Theorem \ref{thm0}.

\section{Proof of Theorem \ref{thm1}}
By Theorem D, we know that if $b_1\in \{3,5,6,9,10\}$, then there is no sequence of positive integers $A=\{a_1<a_2<\cdots\}$ for which
 $P(A)=\mathbb{N}\setminus B$.
 Now, it is sufficient to consider positive integers sequence $B=\{1<b_1<b_2<\cdots\}$ with $b_1\in\{4,7,8\}\cup \{b: b\geq 11, b\in \mathbb{N}\}$.

 Assume that there exists a sequence $A=\{a_1<a_2<\cdots\}$ of positive integers such that
 $P(A)=\mathbb{N}\setminus B$. By the proof of [2, Theorem 1], there exists $A_1=\{a_1<a_2<\ldots<a_k\}\subseteq [1,b_1-1]$ such that
 $$P(A_1)=[0,b_1-1].$$
By Lemma \ref{lem2}, we have
$$P(\{a_1,\ldots,a_{k+2}\})=[0, a_{k+2}+2b_1] \setminus \{b_1, a_{k+2}+b_1\}.$$
 We divide into two cases:

{\bf Case 1.} $b_2=3b_1+3$. If $a_{k+2}\geq b_1+3$, then $b_2\in [0, a_{k+2}+2b_1]$. Since $b_2\not\in P(\{a_1,\ldots,a_{k+2}\})$, we have
$b_2=a_{k+2}+b_1$. Thus $$a_{k+2}=b_2-b_1=2b_1+3>2b_1+1.$$
By Lemma \ref{lem2}, it is impossible.
Thus $a_{k+2}=b_1+2$ and $$P(\{a_1,\ldots,a_{k+2}\})=[0,3b_1+2]\setminus \{b_1,2b_1+2\}.$$
Hence
$$a_{k+3}+P(\{a_1,\ldots,a_{k+2}\})=[a_{k+3}, a_{k+3}+3b_1+2]\setminus \{a_{k+3}+b_1, a_{k+3}+2b_1+2\}.$$

If $a_{k+3}\geq 2b_1+3$, then $2b_1+2\not\in P(A)$, thus $b_2=2b_1+2$, a contradiction.
Hence $a_{k+3}\leq 2b_1+2$.

Since $a_{k+3}>a_{k+2}$, we have $a_{k+3}\geq b_1+3$, thus $b_1+a_{k+3}\neq 2b_1+2$ and
$$P(\{a_1,\ldots,a_{k+3}\})=[0, a_{k+3}+3b_1+2]\setminus \{b_1, a_{k+3}+2b_1+2\}.$$
Since $b_2=3b_1+3\in [0, a_{k+3}+3b_1+2]$ and $b_2\not\in P(\{a_1,\ldots,a_{k+3}\})$, we have $$b_2=3b_1+3=a_{k+3}+2b_1+2\geq 3b_1+5,$$ a contradiction.

{\bf Case 2.} $b_2=3b_1+2$. Since $a_{k+2}\geq b_1+2$, then $b_2\in [0, a_{k+2}+2b_1]$. Since $b_2\not\in P(\{a_1,\ldots,a_{k+2}\})$, we have
$b_2=a_{k+2}+b_1$. Thus $$a_{k+2}=b_2-b_1=2b_1+2>2b_1+1.$$ By Lemma \ref{lem2}, it is impossible.

This completes the proof of Theorem \ref{thm1}.

\section{Proof of Theorem \ref{thm2}}

First, we shall prove that the following result:

\noindent{\bf Fact II} {\it There exists a set sequence $\{A_k\}_{k=2}^{\infty}$ such that

\noindent(i) $A_2\subseteq A_3\subseteq\ldots$;

\noindent(ii) $P(A_k)=[0, 2b_1+(2^{k}-1)d]\setminus \{b_i: \; i=1,\ldots,2^k\}$.}
\vskip 3mm
For $b_1\in \{4,7,8\}\cup \{b: b\geq 11\}$, by the proof of [2, Theorem 1], there exists $A_1=\{a_1<a_2<\ldots<a_k\}\subseteq [1,b_1-1]$ such that
 $$P(A_1)=[0,b_1-1].$$
By Lemma \ref{lem2}, we have $a_{k+2}\leq 2b_1+1$ and
$$P(A_1\cup \{b_1+1\})=[0, 2b_1]\setminus \{b_1\}.$$

Since $a_{k+2}\leq 2b_1+1$ and $b_1+2\leq d\leq 2b_1+1$, we can choose $a_{k+2}=d$, thus
$$P(A_1\cup \{b_1+1,d\})=[0, 2b_1+d]\setminus\{b_1,b_1+d\}.$$
Choose $a_{k+3}=2d$, then
$$a_{k+3}+P(\{a_1,\ldots,a_{k+2}\})=[2d, 2b_1+3d]\setminus \{b_1+2d,b_1+3d\}.$$
Write $$A_2=A_1\cup \{b_1+1,d,2d\}.$$
Since $d\geq b_1+2>b_1$, we have $b_1+2d>2b_1+d$, thus
$$P(A_2)=[0, 2b_1+3d]\setminus \{b_1,b_2,b_3,b_4\}.$$

We have proved that Fact II (ii) is true for $k=2$. Suppose that Fact II is true for $k(\geq 2)$. Now we consider the case $k+1$.

Since $$A_k=A_1\cup \{b_1+1,d, \ldots, 2^{k-1}d\},$$
$$P(A_k)=[0, 2b_1+(2^{k}-1)d]\setminus \{b_i: \; i=1,\ldots,2^k\},$$
we have
$$2^{k}d+P(A_k)=[2^{k}d, 2b_1+(2^{k+1}-1)d]\setminus \{b_i+2^kd: \; i=1,\ldots,2^{k}\}.$$
Write $$A_{k+1}=A_{k}\cup \{2^{k}d\}.$$
Since $$b_i+2^{k}d=b_{2^k+i}, \quad i=1,\ldots,2^k,$$
 we have $$P(A_{k+1})=[0, 2b_1+(2^{k+1}-1)d]\setminus \{b_i: \; i=1,\ldots,2^{k+1}\}.$$

Second, put $$A=\bigcup_{k=2}^{\infty}A_k.$$
If $n\in P(A)$, let $n\leq 2^{k-1}d$, then, by
$$A\setminus A_i\subseteq [2^{k-1}d+1,+\infty)$$ for all $i\geq k$,
we have $n\in P(A_{k})$. By Fact II (ii) we have $n\not\in \{b_1,\ldots,b_{2^k}\}$. Since $n\leq 2^{k-1}d<b_{2^k}$, we have $n\not\in B$. That is, $n\in \mathbb{N}\setminus B$.

Conversely, if $n'\in \mathbb{N}\setminus B$, then $n'\not\in B$, let $n'<b_{2^{k'}}$, we have $n'\not\in \{b_1,\ldots,b_{2^{k'}}\}$. By Fact II (ii) we have $n'\in P(A_{k'})$. So $n'\in P(A)$.

Hence $P(A)=\mathbb{N}\setminus B$.

This completes the proof of Theorem \ref{thm2}.

\end{document}